\documentclass[12pt,  reqno]{amsart}
\usepackage{amsmath,amssymb,amsfonts, amscd,   
pifont, amsthm,  amsxtra, latexsym, mathrsfs,   eufrak}
\makeindex
\usepackage[dvips,hyperindex,breaklinks]{hyperref}
\hypersetup{linkbordercolor=0.2 1 0.5}
\hypersetup{linkbordercolor=0.2 1 0.5}
\usepackage{geometry, txfonts}
\usepackage{array}
\usepackage{epic,eepic}

\numberwithin{equation}{section} 
\theoremstyle{plain}
\newtheorem{thm}{Theorem}[section]

\newtheorem{lem}[thm]{Lemma}

\newtheorem{fact}[thm]{Fact}
\newtheorem*{stra}{Strategy}

\newcommand{\cal}{\mathscr}

\def\g{\mathfrak{g}}

\def\v{\mathfrak{v}}
\def\z{\mathfrak{z}}

\newcommand\C{{\mathbb{C}}}

\newcommand\N{{\mathbb{N}}}

\newcommand\R{{\mathbb{R}}}

\newcommand\HH{{\mathbb{H}}}

\def\cotg{\operatorname{cotg}}

\def\C{{\mathbb C}}

\def\N{{\mathbb N}}

\def\0{{\bar 0}}
\def\1{{\bar 1}}

\def\ad{{\operatorname{ad}}}

\newcommand{\sh}{\operatorname{sinh}}

\def\phi{{\varphi}}

\input xy
\xyoption{all}
\usepackage{enumerate}
\usepackage{float}
\usepackage{graphicx}
\renewcommand{\mathcal}{\mathscr}
\newcommand{\bigsearrow}{\rotatebox[origin=b]{-45}{$\xrightarrow{\kern10mm}$}}
\newcommand{\bigswarrow}{\rotatebox[origin=b]{45}{$\xleftarrow{\kern10mm}$}}
\newcommand{\bignearrow}{\rotatebox[origin=b]{-315}{$\xrightarrow{\kern10mm}$}}
\begin{document}
\title{Uniqueness  of solutions to the Schr{\"o}dinger equation\\
 on the Heisenberg group}
\author{Salem Ben Sa{\" \i}d}
\address{S. Ben Sa{\"\i}d:  Universit\'e Henri Poincar\'e-Nancy 1, 
Institut de Math\'ematiques Elie Cartan, 
B.P. 239, 54506 Vandoeuvre-Les-Nancy, Cedex, France} 
\email{Salem.BenSaid@iecn.u-nancy.fr}
\author{Sundaram Thangavelu}
\address{S. Thangavelu: Department of Mathematics, Indian Institute of 
Science, Bangalore 560 012, India}
\email{veluma@math.iisc.ernet.in}
\begin{abstract} 
This paper deals with the Schr{\"o}dinger equation $i\partial_s u({\bf z},t;s)-\cal L u({\bf z}, t;s)=0,$ where $\cal L$ is the sub-Laplacian on the Heisenberg group. Assume that the initial data $f$ satisfies $| f({\bf z},t)| \leq C q_a({\bf z},t),$ where $q_s$ is the heat kernel associated to $\cal L.$  If in addition $ |u({\bf z},t;s_0)|\leq C q_b({\bf z},t),$ for some $s_0\in \R^*,$ then we prove that $u({\bf z},t;s)=0$ for all $s\in \R $ whenever $ab<s_0^2.$  This result also holds  true on $H$-type groups. 
\end{abstract}
\maketitle
%

\section{Introduction}
Let $\HH^n$ be the $(2n+1)$-dimensional Heisenberg group, and denote by $\cal L$ the sub-Laplacian for 
$\HH^n.$ In this paper we consider the following initial value problem for the Schr{\"o}dinger equation for 
$\cal L:$
\begin{align*}
&i\partial_s u({\bf z},t;s)-\cal L u({\bf z}, t;s)=0,\qquad ({\bf z},t)\in \HH^n,\; s\in \R,\\
&u({\bf z},t;0)=f({\bf z},t)
\end{align*}
where $ f $ is assumed to be in $ L^2(\HH^n).$ Our goal is to find sufficient conditions on the behavior of the solution $u$ at two different times $0$ and $s_0$ which guarantee that $u\equiv 0$   is the unique solution to the above initial data problem.  
More precisely, under some conditions, we prove that  if the function $f$ has sufficient decay and if in addition the solution 
$u({\bf z},t; s_0)$ has sufficient decay at a fixed $s_0\in \R\setminus\{0\},$ then the solution must be trivial. 

Uniqueness theorems of this kind were first proved by Chanillo \cite{Ch} 
where he considered the Schr{\"o}dinger equation associated to the standard 
Laplacian on $ \R^n.$ Using Hardy's theorem for the Euclidean Fourier 
transform he proved a uniqueness theorem for solutions of the 
Schr{\"o}dinger equation. Until then Hardy's theorem was considered only in 
the context of heat equation and Chanillo's work triggered a lot of attention 
on the Schr{\"o}dinger equation. Chanillo himself treated the 
Schr{\"o}dinger equation on complex Lie groups where the initial condition 
was assumed to be $ K$-biinvariant. However, if we use Radon transform the problem can be reduced to the Euclidean case and his result holds without any 
restriction either on the group or on the initial condition.

Similar uniqueness results for other  Schr{\"o}dinger equations and for the Korteweg-de Vries equation have received a good deal of attention in recent 
years (see for instance  \cite{E1, E2, IK, K, R, Z}). These authors have 
developed 
powerful PDE techniques to deal with uniqueness results. Completing a full circle, in a recent work Cowling et al. \cite{C} have used a uniqueness theorem 
for the Schr{\"o}dinger equation to give a `real variable proof' of Hardy's 
theorem. See also the works \cite{E3, E4}.

In this article we prove a uniqueness theorem for the Schr{\"o}dinger equation on the Heisenberg group which is similar to what Chanillo has 
proved in the Euclidean case. Our approach  uses Hardy's theorem for the 
Hankel transform obtained in 
\cite{Tu}, which says that a function and its Hankel transform both cannot 
have arbitrary Gaussian decay at infinity unless, of course, the function is 
identically zero. It is interesting to note that we do not need to use 
Hardy's theorem for the Heisenberg group proved in \cite{Th}.

In the last section we extend our main result to a class of groups that generalizes the Heisenberg group, namely 
$H$-type groups. This class was introduced in \cite{Ka}. The list of  $H$-type groups includes the Heisenberg groups and their analogues built up with quaternions or octonions in place of complex numbers, as well as many other groups. 

\section{Background} 
The $(2n+1)$-dimensional Heisenberg group, denoted by $\HH^n,$ is $\C^n\times \R$ equipped with the group law
$$({\bf z},t)({\bf w},s)=({\bf z}+{\bf w}, t+s+{1\over 2} {\rm Im}({\bf z}\cdot \bar {\bf w})).$$ Under this multiplication $\HH^n$ becomes a nilpotent unimodular Lie group, the Haar measure being the Lebesgue measure $d{\bf z}dt$ on $\C^n\times \R.$ The corresponding Lie algebra   is generated by the vector fields
\begin{align*}
&X_j:={{\partial}\over {\partial x_j}}+{1\over 2} y_j{{\partial}\over {\partial t}},\qquad j=1,2,\ldots,n,\\
&Y_j:={{\partial}\over {\partial y_j}}-{1\over 2} x_j{{\partial}\over {\partial t}},\qquad j=1,2,\ldots,n,
\end{align*}
and $\displaystyle{T:={{\partial}\over {\partial t}}}.$ The sub-Laplacian 
$$\cal L:=-\sum_{j=1}^n X_j^2+Y_j^2$$ can be written as 
$$\cal L=-\Delta_{\R^{2n}}-{1\over 4} |{\bf z}\vert ^2  \partial_t^2+N \partial_t,$$ where $$N=\sum_{j=1}^n x_j {{\partial}\over {\partial y_j}}-y_j {{\partial}\over {\partial x_j}}.$$ This second order differential operator $\cal L$ is hypoelliptic, self-adjoint and nonnegative. It generates a semigroup with kernel $q_s({\bf z}, t),$ called the heat kernel. In particular, $q_s({\bf z}, t)$ is nonnegative and has the property 
$$q_{r^2 s}({\bf z}, t) = r^{-2(n+1)}  q_s(r^{-1}{\bf z}, r^{-2}t),\qquad r\not = 0.$$
Moreover,
$$\int_{\R} e^{i\lambda t} q_s({\bf z},t) dt=(4\pi)^{-n} \left( {{\lambda}\over {\sh \lambda s}}\right)^n e^{-{1\over 4} \lambda (\coth s\lambda ) | {\bf z}| ^2}$$ 
(see \cite{Th}). Henceforth, for $f\in L^1(\HH^n)$ and $\lambda\in \R,$ we will write $$f^\lambda({\bf z}):=\int_{\R} e^{i\lambda t} f({\bf z}, t) dt. $$

We now collect some properties of the heat kernel $ q_s({\bf z},t).$

\begin{fact}\label{fact1} The heat kernel satisfies the  semigroup 
property $q_a\ast q_b({\bf z},t)$ =$q_{a+b}({\bf z},t).$
\end{fact}

The following is a slight modification of \cite[Proposition 2.8.2]{Th}.
\begin{fact}\label{fact2} The heat kernel $q_s({\bf z},t)$ satisfies the following estimate 
\begin{equation}\label{hk}
q_s({\bf z},t) \leq C s^{-n-1}e^{-{{\pi}\over 2} {{|t|}\over s}} e^{-{1\over 4}{{|{\bf z}|^2}\over s}},\qquad s>0.
\end{equation} 
\end{fact}
Indeed, for $s=1$ by \cite[(2.8.9-2.8.10)]{Th}, we have 
$$
q_1({\bf z},t)\leq C e^{-{{\pi}\over 2} |t|} e^{-{1\over 4}|{\bf z}|^2}.
$$
 Now Fact \ref{fact2} follows from the fact that 
$q_s({\bf z},t)=s^{-n-1} q_1(s^{-1/2} {\bf z}, s^{-1} t)$ for all $s>0. $

Let $f$ and $g$ be two functions on $\HH^n.$ The convolution of $f$ with $g$ is defined by 
$$(f\ast g)({\bf z}, t) =\int_{\HH^n} f(({\bf z}, t)(-{\bf w}, s) )g({\bf w}, s) d{\bf w}ds.$$ An easy calculation shows that 
$$(f\ast g)^\lambda({\bf z})=\int_{\C^n} f^\lambda({\bf z}-{\bf w}) g^\lambda({\bf w}) e^{i{\lambda\over 2}{\rm Im}({\bf z}\cdot \bar{\bf w})} d{\bf w}.$$ The right hand side is called the $\lambda$-twisted convolution of $f^\lambda$ with $g^\lambda$ denoted by $f^\lambda\ast_\lambda g^\lambda. $

Let $\cal P$ be the set of all polynomials of the form $ P({\bf z})=\sum_{|  \alpha |+| \beta| \leq m} a_{ \alpha, \beta} {\bf z}^{\alpha} \bar {\bf z}^{\beta}.$ For each pair of nonnegative integers $(p,q),$ we define $\cal P_{p,q}$ to be the subspace of $\cal P$ consisting of all polynomials of the form $P({\bf z})=\sum_{|  \alpha |=p } \sum_{|  \beta |=q }  a_{ \alpha, \beta} {\bf z}^{\alpha} \bar {\bf z}^{\beta}.$ 

Let  $\cal H_{p,q}:=\{ P\in \cal P_{p,q}\;|\; \Delta P=0\},$ where $\Delta$ denotes the Laplacian on $\C^n.$ The elements of $\cal H_{p,q}$ are called bigraded solid harmonics of degree $(p,q).$ We will denote by  $\cal S_{p,q}$ the space of all restrictions of bigraded solid harmonics of degree $(p,q) $ to the sphere $S^{2n-1} . $   By \cite{Th}, the space $L^2(S^{2n-1})$ is the orthogonal direct sum of the spaces $\cal S_{p,q},$ with $p,q\geq 0.$ We choose an orthonormal basis   $\{Y_{p,q}^j\;|\; 1\leq j\leq d(p,q)\}$   for $\cal S_{p,q}. $ Then by standard arguments it follows that every continuous function $f$ on $\C^n$ can  be expanded as 
$$f( r\omega) =\sum_{p,q\geq 0}\sum_{j=1}^{d(p,q)} f_{p,q,j}(r) 
{Y_{p,q}^j}(\omega), \qquad r>0,\; \omega \in S^{2n-1},$$ where 
\begin{equation}\label{sc}
f_{p,q,j}(r) :=\int_{S^{2n-1}} f(r\omega) \overline{Y_{p,q}^j(\omega)} d\sigma(\omega).
\end{equation}

For $k\in \N,$ we write $L_k^{n-1}$ for the Laguerre polynomial defined by 
$$L_k^{n-1}(t)=\sum_{j=0}^{k} {{(-1)^j \Gamma(n+k)}\over{(k-j)! \Gamma(n+j)}}t^j.$$ 
 For $\lambda\in \R^*,$ define the Laguerre functions  $\varphi_{k,\lambda}^{n-1}$ by
\begin{equation}\label{hf}
\varphi_{k,\lambda}^{n-1}({\bf z}) =L_k^{n-1} \left({{|\lambda|}\over 2}\vert {\bf z}\vert^2\right) e^{-{{|\lambda|}\over 4} \vert {\bf z}\vert^2},
\end{equation}
 for ${\bf z}\in \C^n.$ Suppose that $f$ if a radial function in $L^1(\HH^n).$ Then $f(r)$ is in $L^1(\R^+, r^{2n-1} dr),$ where $f(r)$ stands for $f({\bf w})$ with $|{\bf w}|=r. $ For the following Hecke-Bochner formula we refer to   \cite[Theorem 2.6.1]{Th}. 
\begin{thm} \label{HB} Let $f({\bf z})= P({\bf z}) g(|{\bf z}|),$ where $P\in \cal H_{p,q}$ and $g\in L^1(\R^+, r^{2n-1}dr).$ Then for $\lambda\in \R^*,$ we have
$$
 f\ast_\lambda \varphi_{k,\lambda}^{n-1}({\bf z})= (2\pi)^{-n} |\lambda|^{p+q}  P({\bf z}) g\ast_\lambda \varphi_{{k-p}, \lambda}^{n+p+q-1}({\bf z}), $$where  the convolution on the right hand side is taken on $ \C^{n+p+q} $ 
treating the radial functions $ g $ and $ \varphi_{{k-p}, \lambda}^{n+p+q-1}$ 
as functions on $ \C^{n+p+q} .$ More explicitly we have 
\begin{eqnarray}\label{hb}
 g\ast_\lambda \varphi_{{k-p}, \lambda}^{n+p+q-1}({\bf z}) &=&
 {{ (2\pi)^{n+p+q} |\lambda|^{{n+p+q}\over 2} 
{ {2^{-(n+p+q)+1} \Gamma(k-p+1)}}\over {\Gamma(k+n+q)}}  }\\
&& \Big( \int_0^\infty g(s) L_{k-p}^{n+p+q-1}\Big( {{|\lambda|}\over 2} s^2\Big)  e^{-{{|\lambda|}\over 4} s^2} s^{2(n+p+q)-1}ds\Big) L_{k-p}^{n+p+q-1}\left( {{|\lambda|}\over 2} |{\bf z}|^2\right)  
e^{-{{|\lambda|}\over 4} |{\bf z}|^2}  .\nonumber
\end{eqnarray}
\end{thm}

To end this section, let us recall Hardy's uncertainty principle for the 
Hankel transform. For $\alpha>-{1\over 2}$ and $F\in S(\R^+),$ the Hankel transform of order $ \alpha$ is defined by 
\begin{equation}\label{hankel}
\cal H_\alpha  F(s)=\int_0^\infty  F(r)   {{J_\alpha( rs)}\over {(rs)^\alpha}} r^{2\alpha+1} dr,
\end{equation}
 where $ J_\alpha(w) $ is the Bessel function of order $ \alpha $ defined by
$$  J_\alpha(w)=\Big({w\over 2}\Big)^{\alpha}  \sum_{k=0}^\infty {{(-1)^k\left({w \over 2}\right)^{2k}}
 \over { k ! \Gamma(\alpha+k+1)}}.$$ 
 \begin{thm}\label{hardy} (Hardy's theorem, \cite{Tu}) Let $F$  be a measurable function on $\R^+$ such that
 $$ F(r)=O(e^{-a r^2}),\quad \cal H_\alpha F(s)=O(e^{-b s^2})$$ for some positive $a$ and $b.$ Then $F=0$ whenever $\displaystyle{ab>{1\over 4}}$ and 
$ F(r) = C e^{-a r^2} $ whenever $ \displaystyle{ab = {1\over 4}.} $ 
  \end{thm}
 
\section{Schr{\"o}dinger equation on $\HH^n\times \R$}

Let us consider the Schr{\"o}dinger equation on $\HH^n\times \R$
$$i\partial_s u({\bf z},t; s) = \cal L u({\bf z},t;s),$$ with the initial condition $u({\bf z},t;0)=f({\bf z},t).$ As the closure of $\cal L$ on $C_c^\infty (\HH^n)$ is a self-adjoint operator, $-i\cal L$ generates a unitary semi-group $e^{-is\cal L}$ on $L^2(\HH^n),$ and the solution of the above Schr$\rm {\ddot o}$dinger equation is given by  $$u({\bf z},t; s)=e^{-is\cal L} f({\bf z},t).$$

The main result of the paper is:
\begin{thm}\label{main} Let $u({\bf z},t; s)$ be the solution to the Schr{\"o}dinger equation  for the sub-Laplacian $\cal L$ with initial condition $f.$ Suppose that 
\begin{align}
\tag{3.1 a}\label{3.1a}
&\vert f({\bf z},t)  \vert \leq C q_a({\bf z},t),\\
\tag{3.1 b}\label{3.1b}
& \vert u({\bf z},t; s_0)|\leq C q_b({\bf z},t),
\end{align}
for some $a,b>0$ and for a fixed $s_{0} \in \R^*.$ Then $u({\bf z},t; s) =0$ on $\HH^n\times \R$ whenever $ab< s_{0}^2. $
\end{thm}

The remaining part of this section is devoted to the proof of the above statement. 

The heat kernel $ q_s({\bf z},t) $ has an analytic continuation in $ s $ as 
long as real part of $ s $ is positive. However, due to the zeros of the sine 
function, the kernel $  q_{is}({\bf z},t) $ does not exist as can be seen from 
the formula for $  q_s^\lambda({\bf z}).$ Hence the solution 
$u({\bf z},t; s)$ does not have an integral representation. We will therefore 
consider the following regularised problem on $\HH^n\times \R :$
\begin{align*}
&i\partial_s u_\epsilon({\bf z},t; s) = \cal L u_\epsilon({\bf z},t;s),\qquad \epsilon >0,\\
&u_\epsilon({\bf z},t;0)=f_\epsilon({\bf z},t),
\end{align*}
 where $f_\epsilon({\bf z},t) :=e^{-\epsilon \cal L} f({\bf z},t).$ The solution $u_\epsilon$ on $\HH^n\times \R$ is given by 
$$u_\epsilon ({\bf z},t;s)=e^{-is \cal L} f_\epsilon ({\bf z},t)=f\ast q_{\zeta}({\bf z},t),$$ where $\zeta=\epsilon+i s$ and 
$$q_{\zeta}({\bf z},t):={1\over {(8\pi^2 )^n}} \int_\R e^{-i\lambda t} \left({{\lambda}\over {\sinh \lambda \zeta}}\right)^n e^{-{1\over 4} \lambda (\coth \zeta \lambda) |{\bf z}|^2} d\lambda.$$
Observe that the kernel $  q_\zeta({\bf z},t) $ is well defined.

\begin{lem} Under the   assumptions \eqref{3.1a} and \eqref{3.1b}, we have
\begin{align*}
\tag{3.2 a}\label{3.2a}
&\vert f_\epsilon({\bf z},t)  \vert \leq C q_{a+\epsilon}({\bf z},t),\\
\tag{3.2 b}\label{3.2b}
& \vert u_\epsilon({\bf z},t; s_{0})|\leq C q_{b+\epsilon}({\bf z},t).
\end{align*}
\end{lem}
\begin{proof} For the first estimate, we have 
\begin{align*}
\vert f_\epsilon({\bf z},t)  \vert =\vert e^{-\epsilon \cal L} f({\bf z},t)\vert &=\vert f\ast q_\epsilon ({\bf z},t)\vert \\
&\leq  C q_{a+\epsilon}({\bf z},t). 
\end{align*}
Above we have used the fact that $q_s$ is nonnegative and   Fact \ref{fact1}. Similarly we have
 \begin{align*}
 \vert u_\epsilon({\bf z},t; s_{0})| &=| u(\cdot, \cdot\,; s_{0})\ast q_\epsilon  ({\bf z},t)\vert \\
 &\leq C q_{b+\epsilon}({\bf z},t). 
 \end{align*}
\end{proof}

Recall that for $\lambda\in \R,$ the notation $f^\lambda({\bf z})$ stands for the inverse Fourier transform of $f({\bf z},t)$ in the $t$-variable.  
In view of the hypothesis \eqref{3.1a} on $f$ and the estimate  \eqref{hk} on the heat kernel, one can see that the function $\lambda\mapsto f^\lambda({\bf z})$ extends to a holomorphic  function of $\lambda$ on the strip 
$|{\rm Im}(\lambda)| < {\pi\over {2a}}.$ Thus the following statement is 
true. 

\begin{lem} Under the hypothesis  \eqref{3.1a} on $f,$ the inverse Fourier transform $f^\lambda({\bf z})$ of $f({\bf z},t)$ in the $t$-variable extends to a holomorphic function of $\lambda$ in a tubular neighborhood in $\C$ of the real line.
\end{lem}

We point out that the above lemma  also holds for the function $\lambda\mapsto f_\epsilon^\lambda.$

\begin{stra} To prove the main theorem, our strategy is to show that $f=0$ on $\HH^n$ whenever $ab< s_{0}^2.$ However, by the above lemma, showing that $f^\lambda=0$ on $\C^n$ for $0<\lambda<\delta,$ for some $\delta>0,$ will force $f^\lambda=0$ on $\C^n$ for all $\lambda\in \R$ and hence $f=0$ on $\HH^n.$ Furthermore, since $f_\epsilon^\lambda=f^\lambda\ast_\lambda q_\epsilon^\lambda,$ then proving that $f^\lambda =0$ on $\C^n$ for   $0<\lambda<\delta$ is equivalent to show  the same statement for $f_\epsilon^\lambda.$ On the other hand, in order to prove that $f_\epsilon^\lambda({\bf z}) =0$ for  $0<\lambda<\delta,$ for some $\delta >0,$ it is enough to prove  that the spherical harmonic 
coefficients 
$$(f_\epsilon^\lambda)_{p,q,j}(r)=\int_{S^{2n-1}} f_\epsilon^\lambda(r\omega) \overline{Y_{p,q}^j(\omega)} d\sigma(\omega)$$ vanish for  $0<\lambda<\delta,$ for all $p,q\geq 0$ and $1\leq j\leq d(p,q).$   In conclusion,  the proof of  the main theorem reduces to prove that if $ab<  s_{0}^2,$ then $(f_\epsilon^\lambda)_{p,q,j} =0$  on $\R^+$ for  $0<\lambda<\delta,$   for all $p,q\geq 0$ and $1\leq j\leq d(p,q).$
\end{stra}

The following theorem will be of crucial importance to us.   
\begin{thm} \label{crucial}Let us fix $p_{0}, q_{0} \geq 0$ and $1\leq j_{0}
\leq d(p_{0}, q_{0}).$ For all $r>0,$ there exists a constant $c_\lambda$ which depends only on $\lambda$ such that  
$$\int_{S^{2n-1}} u_\epsilon^\lambda(r\omega; s_{0}) \overline{Y_{p_{0},q_{0}}^{j_{0}}(\omega)} d\sigma(\omega)=
c_\lambda r^{p_{0}+q_{0}} e^{i{{ \lambda }\over 4} r^2 \cotg ( \lambda s_{0})}
\cal H_{n+p_{0}+q_{0}-1}\left( e^{i{{ \lambda }\over 4}(\cdot)^2 \cotg ( \lambda s_{0})}  (f_\epsilon^\lambda)_{p_{0},q_{0},j_{0}}^{\sim} \right) \left( {{\lambda r}\over {2\sin (\lambda s_{0})}}\right),$$ 
where $u_\epsilon^\lambda( {\bf z};s_{0})$ denotes the inverse Fourier transform of $u_\epsilon ( {\bf z} , t;s_{0})$ in the $t$-variable,  $\cal H_\alpha$ denotes the Hankel transform of order $\alpha$ (see \eqref{hankel}), and $(f_\epsilon^\lambda)_{p_{0},q_{0},j_{0}}^{\sim}(t):=t^{-(p_0 +q_0)} (f_\epsilon^\lambda)_{p_{0},q_{0},j_{0}}(t).$
\end{thm}
\begin{proof} In what follows $c_\lambda$ will stand for constants depending 
 only on $\lambda$ which will vary from one line to another. Using  Fact \ref{fact2}  we can rewrite $u_\epsilon^\lambda( {\bf z};s_{0})$ as $$u_\epsilon^\lambda( {\bf z};s_{0})=f_\epsilon^\lambda\ast_\lambda q_{is_{0}}^\lambda({\bf z}),$$ where $$q_{is_{0}}^\lambda({\bf z})=(4\pi)^{-n} \left( {{\lambda}\over {i\sin \lambda s_{0}}}\right)^n e^{{i\over 4} \lambda (\cotg \lambda s_{0})| {\bf z}|^2} $$ which exits for all but a countably many values of $ \lambda.$ Thus  
\begin{align*}
&\int_{S^{2n-1}} u_\epsilon^\lambda(r\omega; s_{0}) \overline{Y_{p_{0},q_{0}}^{j_{0}}(\omega)} d\sigma(\omega)\\
&= \int_{S^{2n-1}} \Big[ \int_{\C^n} f_\epsilon^\lambda (r\omega  -{\bf w}) q_{is_{0}}^\lambda({\bf w}) e^{i{\lambda\over 2} {\rm Im}( r \omega\cdot \bar {\bf w})} d{\bf w}\Big]  \overline{Y_{p_{0},q_{0}}^{j_{0}}(\omega)} d\sigma(\omega)\\
&= \int_{S^{2n-1}} \Big[ \int_{\C^n} f_\epsilon^\lambda ({\bf w}) q_{is_{0}}^\lambda(r\omega  -{\bf w}) e^{-i{\lambda\over 2} {\rm Im}(r\omega\cdot \bar {\bf w})} d{\bf w}\Big]  \overline{Y_{p_{0},q_{0}}^{j_{0}}(\omega)} d\sigma(\omega).
\end{align*}
We now expand $f_\epsilon^\lambda$ in terms of bigraded spherical harmonics as $$f_\epsilon^\lambda(t \eta)=\sum_{p,q\geq 0} \sum_{j=1}^{d(p,q)} (f_\epsilon^\lambda)_{p,q,j}(t)  Y_{p,q}^j (\eta), $$ where 
$(f_\epsilon^\lambda)_{p,q,j}$ is as in \eqref{sc}. Further, by \cite[(2.8.7)]{Th} we have
$$q_{is_{0}}^\lambda( r\omega-t\eta)=(2\pi)^{-n} |\lambda|^n \sum_{k=0}^\infty e^{-i(2k+n) |\lambda | s_{0}}\varphi_{k,\lambda}^{n-1}(r\omega-t \eta),$$ where $\varphi_{k,\lambda}^{n-1}$ is given by \eqref{hf}.
Now the Hecke-Bochner formula for the $\lambda$-twisted convolution (see Theorem  \ref{HB}) gives us 
\begin{align*}
&\int_0^\infty \int_{S^{2n-1}} (f_\epsilon^\lambda)_{p,q,j}(t) Y_{p,q}^j(\eta) \varphi_{k,\lambda}^{n-1}
(r\omega-t\eta)e^{-i {\lambda\over 2} rt {\rm Im}(\omega\cdot \bar \eta)} t^{2n-1} dt d\sigma(\eta)\\
&=\int_0^\infty \int_{S^{2n-1}} (f_\epsilon^\lambda)_{p,q,j}^{\sim}(t) P_{p,q}^j(t \eta) \varphi_{k,\lambda}^{n-1}
(r\omega-t\eta)e^{-i {\lambda\over 2} rt {\rm Im}(\omega\cdot \bar \eta)} t^{2n-1} dt d\sigma(\eta)\\
&=
\Big[  (f_\epsilon^\lambda)_{p,q,j}^{\sim}\,P_{p,q}^j \Big] \ast_{-\lambda} \varphi_{k,\lambda}^{n-1}(r\omega)\\
&=(2\pi)^{-n} |\lambda|^{p+q} P_{p,q}^j(r\omega) \Big[ (f_\epsilon^\lambda)_{p,q,j}^{\sim} \ast_{-\lambda} \varphi_{k-p,\lambda}^{n+p+q-1}\Big](r\omega), 
\end{align*}
where the convolution on the right hand side is on $\C^{n+p+q}.$ Here $P_{p,q}^j( r\omega):=r^{p+q} Y_{p,q}^j(\omega)$ and $F^{\sim} (t) =t^{-(p+q)} F(t).$ Above we have used the fact that  $\varphi_{k,\lambda}^{n-1} =\varphi_{k,-\lambda}^{n-1}.$ Using the orthogonality of the basis 
$\{Y_{p,q}^j\;:\; 1\leq j\leq d(p,q)\}$ we obtain:
\begin{align*}
&\int_{S^{2n-1}} \Big[ \int_{\C^n} f_\epsilon^\lambda ({\bf w}) q_{is_{0}}^\lambda(r\omega  -{\bf w}) e^{-i{\lambda\over 2} {\rm Im}(r\omega\cdot \bar {\bf w})} d{\bf w}\Big]  \overline{Y_{p_{0},q_{0}}^{j_{0}}(\omega)} d\sigma(\omega)\\
&=c_\lambda  r^{p_{0}+q_{0}} \sum_{k\geq p_{0}} e^{-i(2k+n) |\lambda| s_{0}} 
(f_\epsilon^\lambda)_{p_{0},q_{0},j_{0}}^{\sim} \ast_{-\lambda} \varphi_{k-p_{0},\lambda}^{n+p_{0}+q_{0}-1}(r\omega)\\
&=c_\lambda  r^{p_{0}+q_{0}} \sum_{k=0}^\infty e^{-i(2k+n+2p_{0}) |\lambda| s_{0}} 
(f_\epsilon^\lambda)_{p_{0},q_{0},j_{0}}^{\sim} \ast_{-\lambda} \varphi_{k ,\lambda}^{n+p_{0}+q_{0}-1}(r\omega). 
\end{align*}
On the other hand, by \eqref{hb} we have 
\begin{align*}
& (f_\epsilon^\lambda)_{p_{0},q_{0},j_{0}}^{\sim} \ast_{-\lambda} \varphi_{k,\lambda}^{n+p_{0}+q_{0}-1} (r\omega) =c_\lambda 
 {{ \Gamma(k+1)}\over {\Gamma(k+n+p_{0}+q_{0})}} \\
&\qquad\left( \int_0^\infty  (f_\epsilon^\lambda)_{p_{0},q_{0},j_{0}}^{\sim} (t) \varphi_{k,\lambda}^{n+p_{0}+q_{0}-1}(t) t^{2(n+p_{0}+q_{0})-1}dt\right) \varphi_{k,\lambda}^{n+p_{0}+q_{0}-1}(r\omega).
\end{align*}
Hence we obtain
\begin{align*}
&\int_{S^{2n-1}} u_\epsilon^\lambda(r\omega; s_{0})  \overline{Y_{p_{0},q_{0}}^{j_{0}}(\omega)} d\sigma(\omega) \\
&= c_\lambda  r^{p_{0}+q_{0}} \sum_{k=0}^\infty {{\Gamma(k+1)}\over {\Gamma( k+n+p_{0}+q_{0})}} e^{-i(2k+n+2p_{0})|\lambda| s_{0}} 
L_k^{n+p_{0}+q_{0}-1}\Big({{|\lambda|}\over 2} r^2\Big) e^{-{{|\lambda|}\over 4} r^2}\\
&\qquad \Big( \int_0^\infty (f_\epsilon^\lambda)^\sim_{p_{0}, q_{0}, j_{0}} (t) \varphi_{k,\lambda}^{n+p_{0}+q_{0}-1}(t) t^{2(n+p_{0}+q_{0})-1} dt\Big)\\
&=c_\lambda r^{p_{0}+q_{0}} \int_0^\infty (f_\epsilon^\lambda)^\sim_{p_{0}, q_{0}, j_{0}} (t) K_{\lambda}( r,t; s_{0}) t^{2(n+p_{0}+q_{0})-1} dt, 
\end{align*}
where 
$$K_\lambda(r,t;s_{0}) :=
\sum_{k=0}^\infty {{\Gamma(k+1)}\over {\Gamma( k+n+p_{0}+q_{0})}}  
e^{-i(2k+n+2p_{0})|\lambda| s_{0}}  e^{-{{|\lambda|}\over 4} (r^2+t^2)}
L_k^{n+p_{0}+q_{0}-1}\left({{|\lambda|}\over 2} r^2\right) 
L_k^{n+p_{0}+q_{0}-1}\left({{|\lambda|}\over 2} t^2\right)  .$$ 
Now we can use the following Hille-Hardy identity  (see for instance \cite{Th1})
$$\sum_{k=0}^\infty {{\Gamma(k+1)}\over {\Gamma(k+\alpha+1)}} L_k^{\alpha} (x) L_k^{\alpha}(y) w^k=(1-w)^{-(\alpha+1)}
  e^{-{w\over {1-w}}(x+y)} \widetilde{J}_\alpha\left({{2(-xyw)^{1/2}}\over {1-w}}\right),$$ where $\widetilde{J}_\alpha(w):=\left({w\over 2}\right)^{-\alpha} J_\alpha(w)$ and $J_\alpha$    is the Bessel function of order $\alpha.$ Thus we may rewrite the kernel $K_\lambda$ as 
$$K_\lambda( r,t;s_{0})=   e^{i |\lambda| s_{0}(q_{0}-p_{0})} (2i\sin (|\lambda| s_{0}))^{-(n+p_{0}+q_{0})} e^{i{{ \lambda }\over 4}(r^2+t^2) \cotg ( \lambda s_{0})} \widetilde J_{n+p_{0}+q_{0}-1}\left( {{\lambda} \over 2} {{rt}\over {\sin (\lambda  s_{0})}}\right).$$
Thus we arrive at 
\begin{align*}
&\int_{S^{2n-1}} u_\epsilon^\lambda(r\omega; s_{0})  \overline{Y_{p_{0},q_{0}}^{j_{0}}(\omega)} d\sigma(\omega) \\ 
&=c_\lambda r^{p_{0}+q_{0}}  \int_0^\infty e^{i{{ \lambda }\over 4}(r^2+t^2) \cotg ( \lambda s_{0})}   (f_\epsilon^\lambda)_{p_{0},q_{0},j_{0}}^{\sim}  (t) 
{\widetilde J}_{n+p_{0}+q_{0}-1}\left( {{\lambda} \over 2} {{rt}\over {\sin (\lambda  s_{0})}}\right)
t^{2(n+p_{0}+q_{0})-1} dt\\
&=c_\lambda r^{p_{0}+q_{0}} e^{i{{ \lambda }\over 4} r^2 \cotg ( \lambda s_{0})}
\cal H_{n+p_{0}+q_{0}-1}\left( e^{i{{ \lambda }\over 4}(\cdot)^2 \cotg ( \lambda s_{0})}  (f_\epsilon^\lambda)_{p_{0},q_{0},j_{0}}^{\sim} \right) \left( {{\lambda r}\over {2\sin (\lambda s_{0})}}\right).
\end{align*}
Hence Theorem \ref{crucial} has been proved.
\end{proof}

We are ready to complete the proof of the main result.

The estimate  \eqref{3.2a} on $f_\epsilon({\bf z},t)$ together with Fact \ref{fact1} lead  us to 
$$| f_\epsilon^\lambda({\bf z}) | \leq c e^{-{1\over 4} {{| {\bf z} |^2}\over {a+\epsilon}}},$$ for some constant $c.$ Thus, the spherical harmonic 
coefficient     $(f_\epsilon^\lambda)_{p_{0},q_{0},j_{0}}^{\sim}$ satisfies 
$$| (f_\epsilon^\lambda)_{p_{0},q_{0},j_{0}}^{\sim} (t)| \leq c t^{-(p_{0}+q_{0})} e^{-{1\over 4} {{t^2}\over {a+\epsilon}}}.$$ 
On the other hand, by means of Theorem \ref{crucial} and the estimate \eqref{3.2b}   on $u_\epsilon({\bf z},t;s_{0}),$ we deduce that  
 $$  \left | \cal H_{n+p_{0}+q_{0}-1}\left( e^{i{{ \lambda }\over 4}(\cdot)^2 \cotg ( \lambda s_{0})}  (f_\epsilon^\lambda)_{p_{0},q_{0},j_{0}}^{\sim} \right) \left( {{\lambda r}\over {2\sin (\lambda s_{0})}}\right)\right | \leq c_\lambda r^{-(p_{0}+q_{0})}     e^{-{1\over 4} {{r^2}\over {b+\epsilon}}}.$$ That is 
 $$   \left | \cal H_{n+p_{0}+q_{0}-1}\left( e^{i{{ \lambda }\over 4}(\cdot)^2 \cotg ( \lambda s_{0})}  (f_\epsilon^\lambda)_{p_{0},q_{0},j_{0}}^{\sim} \right)  (r) \right | \leq c_\lambda   r^{-(p_{0}+q_{0})}     e^{-{1\over 4} \left( {    {2\sin (\lambda s_{0})} \over {\lambda s_{0}} }\right)^2  {{r^2 s_{0}^2}\over {b+\epsilon}}}.$$
Given $a,b>0$ such that $ab< s_{0}^2 $ we can choose $\epsilon >0$ such that $(a+\epsilon)(b+\epsilon) < s_{0}^2.$ We can also choose $\delta >0$ small enough in such a way that for $0<\lambda<\delta$ we have 
   $ (a+\epsilon )(b+\epsilon) < s_{0}^2 \left( {{ \sin (\lambda s_{0})}\over { \lambda s_{0}
}    }\right)^2.$ This inequality can be written as 
$${1\over {4(a+\epsilon)}} {{s_{0}^2}\over { 4(b+\epsilon)}} \left( {{  2\sin (\lambda s_{0})}\over { \lambda s_{0}
}    }\right)^2>{1\over 4}.$$ Therefore, by Hardy's theorem for the Hankel transform (see Theorem \ref{hardy}), we deduce that for $0<\lambda<\delta$ we have  $(f_\epsilon^\lambda)_{p_{0},q_{0},j_{0}}^{\sim} =0,$ for all $p_{0},q_{0}\geq 0$ and $1\leq j_{0}\leq d(p_{0}, q_{0}).$  That is $f_\epsilon^\lambda =0$ on $\C^n$ for $0<\lambda<\delta,$ which forces $f_\epsilon ^\lambda =0$ for all $\lambda$ and hence $f_\epsilon= 0$ on $\HH^n.$  That is  $f= 0$ on $\HH^n.$  This finishes the proof Theorem \ref{main}. 
 
\section{The main result for $H$-type groups}  
Let $\g$ be a two step nilpotent Lie algebra over $\R$ with an inner product $\langle \cdot,\cdot\rangle.$  The corresponding simply connected Lie group is denoted by $G.$  Let $\z$ be the center of $\g$ and $\v$ the orthogonal complement of $\z$ in $\g.$ The Lie algebra $\g$ is called an $H$-type algebra if for every ${\bf v}\in \v,$ the map $\ad_{\bf v}:\v\rightarrow \z$ is a surjective isometry when restricted to the orthogonal complement of its kernel. 

For the $H$-type algebra $\g= \v\oplus \z,$ let $\dim(\v)=2n$ and $\dim(\z)=k.$ The class of groups of $H$-type includes the Heisenberg group $\HH^n$ when $k=1.$ Let $\eta$ be a unit element in $\z$ and denote its orthogonal complement 
in $\z$ by $\eta^\perp. $ The quotient algebra $\g/\eta^\perp$ is a Lie 
algebra with Lie bracket $[ X, Y]_\eta=\langle [X,Y], \eta\rangle.$ 

The quotient  $\g/\eta^\perp$  is an $H$-type algebra with inner product $\langle\cdot,\cdot \rangle _\eta$ given by 
$\langle ({\bf v}_1, { t}_1),  ({\bf v}_2, {  t}_2)\rangle_\eta=\langle {\bf v}_1,{\bf v}_2\rangle + t_1t_2,$ where ${\bf v}_1,{\bf v}_2\in \v,$  $t_1,t_2\in \R,$ and $ \langle {\bf v}_1,{\bf v}_2\rangle$ is the inner product in $\g.$  Here  $({\bf v} , { t} )$ stands for the coset of  ${\bf v}+  { t}\eta$ in $\g/\eta^\perp.$ Moreover, if we denote by $G_\eta$ the simply connected Lie group with Lie algebra  $\g/\eta^\perp$, then by \cite{Ri}, the Lie group $G_\eta$ is  isomorphic to the Heisenberg group $\HH^n=\C^n\times \R.$ We refer to \cite{B} for more details on the theory of $H$-type groups.
 
We fix an orthonormal basis $X_1,\ldots, X_{2n}$ for $\v,$ and define the sub-Laplacian  by 
$$\cal L=-\sum_{j=1}^{2n}X_j^2.$$ It is known that $\cal L$ generates a semigroup   which is given by convolution with the heat kernel for $G.$ As in the case of Heisenberg group, the kernel is explicitly known and is given by 
$$h_s({\bf v}, {\bf t})={1\over {2^n (2\pi)^{n+k/2}}} \int_0^\infty 
{ {\lambda^{k/2}}\over{ | {\bf t}|^{{k/2}-1} }} J_{k/2-1}(\lambda | {\bf t} | ) \left( {{\lambda}\over {\sh(s\lambda)} }\right)^{n} e^{-{1\over 4} \lambda (\coth s \lambda) |{\bf v}|} d\lambda,$$ for $ ({\bf v}, {\bf t}) \in G$ and $s>0.$ Here $J_\alpha$ denotes the Bessel function of order $\alpha.$ This formula has been proved in \cite{Ra}, where the author also obtains the integral expression for the analytic continuation, $h_\zeta,$ of the heat kernel $h_s$ as long as ${\rm Rel}(\zeta)>0.$ 
 
We now consider the solution of the Schr{\rm \"o}dinger equation  on $G\times \R$
\begin{align*}
&i\partial_s u({\bf v}, {\bf t}; s) =\cal L u({\bf v}, {\bf t};s),\\
&u({\bf v}, {\bf t}; 0)=f({\bf v}, {\bf t}), 
\end{align*}
which is given by $u({\bf v}, {\bf t};s) =e^{-i s\cal L} f({\bf v}, 
{\bf t}).$ When we replace the initial condition $f$ by $e^{-\epsilon \cal L} f,$ for some $\epsilon >0,$ then the solution is given by 
$$u_\epsilon({\bf v}, {\bf t};s) =f\ast h_\zeta ({\bf v}, {\bf t}),  \qquad \zeta =\epsilon+i s. $$

We claim that the uniqueness Theorem \ref{main} for the Schr{\rm \"o}dinger equation on $\HH^n\times \R$ is true in the more general setting $G\times \R.$ The rest of this section is devoted to the proof of the following theorem.

\begin{thm}\label{Htype} Let $u({\bf v}, {\bf t}; s)$ be the solution of the 
Schr{\rm \"o}dinger equation on  $G\times \R,$ with initial data $f.$  
Assume that  $| f({\bf v}, {\bf t})| \leq C h_a  ({\bf v}, {\bf t}) $ for 
some $a>0.$ Further, suppose that there exists  $s_0 \in \R\setminus\{0\}$ such that  
$|u({\bf v}, {\bf t}; s_0)| \leq C h_b ({\bf v}, {\bf t})$  for some $b>0.$  
If $ab< s_0 ^2,$ then   $u({\bf v}, {\bf t}; s ) =0$ for all $({\bf v}, {\bf t})\in G$ and for all $s\in \R.$
\end{thm}

For a suitable function $f$ on $G$ we define its partial Radon transform  $\cal R_\eta f({\bf v}, {t})$  on $G_\eta$ by 
$$ \cal R_\eta f({\bf v}, {  t}) =\int_{\eta^\perp}   f({\bf v}, { t}\eta +\nu)d\nu$$ where $d\nu$ is the Lebesgue measure on $\eta^\perp.$  Since $G_\eta$ can be identified with the Heisenberg group $ \HH^n,$ we can think of $\cal R_\eta f$ as a function on $\HH^n.$ With this identification it has been proved in \cite{Ra}  that $\cal R_\eta h_s ({\bf v}, {  t}) = q_s({\bf v}, {  t}),$   for $s>0,$ where $q_s ({\bf v}, {  t})$  is the heat kernel from section 2. The above identity between the heat kernels holds true even when $s$ is complex with ${\rm Rel}(s)>0.$ 

In view of the assumptions on $ f({\bf v}, {\bf t})$ and $ 
u({\bf v}, {\bf t}; s_0)$ it follows that $\cal R_\eta f  ({\bf v}, {  t})$ 
and $\cal R_\eta u ({\bf v}, {  t}; s_0 )$ satisfy 
\begin{align*}
&| \cal R_\eta f  ({\bf v}, {  t})| \leq C q_a    ({\bf v}, {  t}),\\
&| \cal R_\eta u ({\bf v}, {  t}; s_0 ) | \leq C q_b ({\bf v}, { t}).
\end{align*}
Moreover, using the fact that under the Radon transform $\cal R_\eta,$ the  sub-Laplacian $\cal L$ on $G$ goes into the sub-Laplacian $\cal L$ on $\HH^n $ (see \cite{Ri}), it follows that $\cal R_\eta u$ solves the Schr{\rm \"o}dinger equation  on $\HH^n\times \R$ with initial data $\cal R_\eta f  ({\bf v}, {  t}).$ Hence we can appeal to Theorem \ref{main} to conclude that $\cal R_\eta u({\bf v}, {  t}; s)=0$ for all $s\in \R$ and for all $\eta \in \z $ whenever 
$ab<s_0^2.$ Now the injectivity of the Radon transform implies that if  
$ab<s_0^2,$ then $u({\bf v}, {\bf t}; s)=0$ for all $({\bf v}, {\bf t})\in G$  and $s\in \R.$ This establishes Theorem \ref{Htype}. 

\section{Some concluding remarks}

It would be interesting to see if Theorem 3.1 is sharp. Though we believe it 
is sharp we are not able to prove it. The main reason for the difficulty lies 
in the fact that  the heat kernel $ q_a({\bf z}, {t}) $ does not have Gaussian 
decay in the central variable. For the same reason   the equality 
case of Hardy's theorem for the group Fourier transform on the Heisenberg 
group is still an open problem. However, if we assume conditions on $ f^\lambda $ and $ u^\lambda $ 
instead of on $ f $ and $ u $ we can prove the following result.

\begin{thm}\label{maineq} Let $u({\bf z},t; s)$ be the solution to the 
Schr{\"o}dinger equation  for the sub-Laplacian $\cal L$ with initial 
condition $f.$ Fix $ \lambda \neq 0 $ and suppose that
$$ \vert f^\lambda({\bf z})  \vert \leq C q_a^\lambda({\bf z}),
\qquad 
 \vert u^\lambda({\bf z}; s_0)|\leq C q_b^\lambda({\bf z})$$
for some $a,b>0 $ and for a fixed $s_{0} \in \R^*.$ Then  we have 
$ f^\lambda({\bf z}) = c_\lambda q_{a}^\lambda({\bf z})
 e^{-i{{ \lambda }\over 4}|{\bf z}|^2 \cotg ( \lambda s_{0})} $ 
 whenever  $ \tanh(a\lambda) \tanh(b\lambda) =  \sin^2 (\lambda s_{0}). $
\end{thm}

To prove this theorem, we can proceed as  in the proof of Theorem 3.1. We end 
up with the estimates
$$   \left | \cal H_{n+p_{0}+q_{0}-1}\left( e^{i{{ \lambda }\over 4}(\cdot)^2 
\cotg ( \lambda s_{0})}  (f^\lambda_{p_{0},q_{0},j_{0}})^{\sim} 
\right)  (r) \right | \leq c_\lambda   r^{-(p_{0}+q_{0})}     
e^{-{{\lambda}\over 4}\coth(\lambda b) \left( {{2\sin (\lambda s_{0})} \over {\lambda s_{0}} }\right)^2  {r^2}} $$
and 
$$ \left |(f^\lambda_{p_{0},q_{0},j_{0}})^{\sim}(r) \right | \leq 
c_\lambda e^{-{{\lambda}\over 4}\coth(\lambda a)r^2}.$$
We can now appeal to the equality case of Hardy's theorem for the Hankel 
transform (Theorem \ref{hardy}) to conclude that
$$ f^\lambda_{p_{0},q_{0},j_{0}}(r) = c_\lambda(p_0,q_0,j_0) r^{p_0+q_0}
e^{-{{\lambda}\over 4}\coth(\lambda a)r^2} 
e^{-i{{ \lambda }\over 4}r^2 \cotg ( \lambda s_{0})}. $$ 
But this is not compatible with the hypothesis on $ f^\lambda $ unless 
$ c_\lambda(p_0,q_0,j_0) = 0 $ for all $ (p_0,q_0) \neq (0,0).$ Hence 
$ f^\lambda $ is radial and equals 
$ c_\lambda q_a^\lambda({\bf z}) 
e^{-i{{ \lambda }\over 4}|{\bf z}|^2 \cotg ( \lambda s_{0})}. $ This 
proves Theorem \ref{maineq}.

The above result can be viewed as a uniqueness theorem for solutions of the 
Schr\"odinger equation associated to the twisted Laplacian $ L_\lambda $ 
defined by $ \cal L (e^{i\lambda t}f({\bf z})) = e^{i\lambda t}L_\lambda 
f({\bf z}).$ Indeed, $ q_a^\lambda({\bf z}) $ is the heat kernel associated to 
this operator. We refer to \cite[(2.3.7)]{Th} for the explicit expression of $L_\lambda.$ 
We can also consider the result as an analogue of Hardy's 
theorem for fractional powers of the symplectic Fourier transform. In fact, 
the unitary operator $ e^{ins}e^{-isL_1} $ with $ s = \frac{\pi}{2} $ is just 
the symplectic Fourier transform. Thus the above theorem for $ s_0 =  
\frac{\pi}{2} $ follows immediately from Hardy's theorem for the Fourier 
transform whereas for other values of $ s_0 $ we require a longwinding proof.

For the sake of completeness we state another result which can be considered 
as a theorem for fractional Fourier transform as well as a theorem for 
solutions of the Schr\"odinger equation associated to the Hermite operator 
$ H = -\Delta +|x|^2$ on $\R^n.$  This elliptic operator generates the Hermite semigroup 
whose kernel is known explicitly. We also know that $ e^{\frac{i}{4}n\pi} e^{
-\frac{i}{4}\pi H} $ is the Fourier transform on $ \R^n.$

\begin{thm} Let $ u(x,s) = e^{-isH}f(x) $ be the solution to the Schr\"odinger 
equation $$ i\partial_s u(x,s) - Hu(x,s)=0, $$ with initial condition $ f.$ Suppose
$$ |f(x)| =O(e^{-a|x|^2}),\qquad  |u(x,s_0)| = O(e^{-b|x|^2}) $$ for some 
$ a, b > 0.$ Then $ u = 0 $ on $\R^n\times \R$ whenever $ a b \sin^2(2s_0) > \frac{1}{4}.$
\end{thm}

The theorem follows from Hardy's theorem for $ \R^n $ once we 
realise $ u $ as the Fourier transform of a function. But this is easy 
to check in view of the Mehler's formula (see \cite{Th}) for the Hermite 
functions. In view of this formula, the kernel of $ e^{-itH} $ is given by
$$ K_r(x,y) = \pi^{-n/2}(1-r^2)^{-n/2}e^{-\frac{1}{2}\frac{1+r^2}{1-r^2}(|x|^2+
|y|^2) + \frac{2r}{1-r^2}x\cdot y} $$ 
where $ r = e^{-2it}.$ Using this formula the theorem can be easily proved.

\begin{center}
{\bf Acknowledgments}
\end{center}

This work was done while the second author was visiting Universit\'e Henri 
Poincar\'e, Nancy. He wishes to thank Salem Ben Said for the invitation and 
the warm hospitality during his visit.

\end{document}